\documentclass[11pt,a4paper]{article}
\usepackage[left=20mm,right=20mm,top=20mm,bottom=25mm]{geometry}
\usepackage{latexsym,amssymb,amsmath,amsthm,float,cite,setspace,lineno}

\newtheorem{theorem}{Theorem}

\newtheorem{lemma}[theorem]{Lemma}

\newtheorem{claim}{Claim}

\theoremstyle{definition}

\usepackage{tikz}
\usetikzlibrary{calc}
\usetikzlibrary{arrows.meta}

\begin{document}
\onehalfspace

\title{Degenerate Vertex Cuts in Sparse Graphs}
\author{Thilo Hartel \and Johannes Rauch \and Dieter Rautenbach}
\date{}
\maketitle
\begin{center}
{\small  
Institute of Optimization and Operations Research\\ 
Ulm University, Ulm, Germany\\
\texttt{thilo.hartel@uni-ulm.de\\
johannes.rauch@mailbox.org\\
dieter.rautenbach@uni-ulm.de}}
\end{center}

\begin{abstract}
For a non-negative integer $k$, 
a vertex cut in a graph is $k$-degenerate if it induces a $k$-degenerate subgraph.
We show that a graph of order $n$ at least $2k+2$ without a $k$-degenerate cut
has the size at least $\frac{1}{2}\left(k+\Omega\left(\sqrt{k}\right)\right)n$
and that a graph of order $n$ at least $5$ without a $2$-degenerate cut 
has the size at least $\frac{27n-35}{10}$.
For $k\geq 2$, 
we show that a connected graph $G$ of order $n$ at least $k+6$
and size $m$ at most $\frac{k+3}{2}n+\frac{k-1}{2}$
has a minimum $k$-degenerate cut.
\end{abstract}

\section{Introduction}

We consider finite, simple, and undirected graphs and use standard terminology.
A graph $G$ is {\it $k$-degenerate} for some non-negative integer $k$
if every non-empty subgraph $H$ of $G$ 
has a vertex $u$ of degree $d_H(u)$ at most $k$.
In particular, a graph is $0$-degenerate if it has no edges,
and $1$-degenerate if it is a forest.
A set $S$ of vertices of a graph $G$ is a {\it cut of $G$} 
if $G-S$ is disconnected.

In the present paper we study the existence of $k$-degenerate cuts in sparse graphs,
which are cuts inducing $k$-degenerate subgraphs. 
Answering a question of Caro,
it was shown by Chen and Yu \cite{chyu} that 
every graph of order $n$ with less than $2n-3$ edges 
has an {\it independent cut}, 
which is a $0$-degenerate cut. 
This seminal result with its elegant inductive proof led to considerable further research,
cf.~\cite{chfaja,chrarare,lepf,rara} and the references therein. 
At the 7th C5 Graph Theory Workshop (Kurort Rathen, 2003) \cite{in},
Atsushi Kaneko proposed the conjecture that 
every graph of order $n$ with less than $3n-6$ edges has a {\it forest cut},
which is a $1$-degenerate cut. 
This conjecture we rediscovered by Chernyshev et al.~\cite{chrara} 
who verified it for planar graphs and showed that 
every graph of order $n$ with less than $(11n-18)/5$ edges has a forest cut.
This bound was improved to $(9n-15)/4$ by Botler et al.~\cite{bocofefigosa} 
and to $(19n-28)/8$ for $n\geq 4$ by Bogdanov et al.~\cite {bonesovoruvo}
but the conjecture remains open.
Bogdanov et al.~also consider the existence of {\it bipartite cuts},
which are cutsets inducing a $2$-colorable subgraph.

Chen and Yu's result and Kaneko's conjecture suggest the conjecture that 
every graph of sufficiently large order $n$ 
with less than 
$$(k+2)n-{k+3\choose 2}$$ 
edges has a $k$-degenerate cut,
which would be best possible in view of the join of
$K_{k+2}$ and $\overline{K}_{n-k-2}$.
Unfortunately, this naive conjecture turns out to be quite false.
Modifying a construction of Bessy et al.~\cite{beraraso}
based on results of Axenovich et al.~\cite{axsesnwe},
Aubian et al.~\cite{aubobofopi} construct graphs of 
arbitrarily large order $n$ 
and less than 
$$\frac{1}{2}\left(k+O\left(\sqrt{k\ln(k)}\right)\right)n$$ 
edges
that do not have a $k$-degenerate cut
or even a cut inducing a $(k+1)$-colorable subgraph.
Note that a graph $G$ of order at least $k+2$ 
that does not have a $k$-degenerate cut
has minimum degree at least $k+2$, 
because, otherwise, 
the neighborhood $N_G(u)$ of a vertex $u$ of minimum degree is a $k$-degenerate cut.
This trivial argument already implies that graphs of sufficiently large order $n$
that do not have a $k$-degenerate cut
have at least $\frac{1}{2}\left(k+2\right)n$ edges.

Our first contribution is the following result.

\begin{theorem}\label{thm:size}
Let $k$ be a positive integer.
If $G$ is a graph of order $n$ at least $2k+2$ without a $k$-degenerate cut, 
then the size of $G$ is at least 
$\frac{1}{2}\left(k+\frac{13\sqrt{k}}{190}-\frac{1}{38}\right)n$.
\end{theorem}
The proof of Theorem \ref{thm:size} only considers vertices 
at distance at most $2$ from small degree vertices,
that is, it relies on a relatively local argument.
Nevertheless, the results of Aubian et al.~\cite{aubobofopi} 
imply that Theorem \ref{thm:size}
is best possible up to a factor of order $\sqrt{\ln(k)}$
at the $\sqrt{k}$-term.

For $k=2$, we improve Theorem \ref{thm:size} as follows.

\begin{theorem}\label{thm2}
If $G$ is a graph of order $n$ at least $5$ without a $2$-degenerate cut, 
then the size of $G$ is at least $\frac{27n-35}{10}$.
\end{theorem}
The old and new results discussed so far concerned $k$-degenerate cuts
of arbitrary order. 
Already Chen et al.~\cite{chfaja} study 
under which conditions there are small independent cuts.
In particular, for $s\leq 3$ and any positive integer $n$,  
they determine the largest value of $m$ 
--- up to a small additive error for $s=3$ ---
such that a graph of order $n$ and size at most $m$
has an independent cut of order at most $s$.
Recently, 
Cheng et al.~\cite{chtazh} consider the existence of minimum cuts 
that are $k$-degenerate for $k\in \{ 0,1\}$.
In particular, they show 
that every connected graph of order $n\geq 7$
and size at most $\left\lfloor\frac{3n}{2}\right\rfloor$
has an independent minimum cut, and
that every connected graph of order $n\geq 7$
and size at most $2n$
has a minimum cut that induces a forest.
In both cases, the bounds on the size are best possible.

Our third result concerns the existence of minimum cuts that are $k$-degenerate.

\begin{theorem}\label{thmmin}
Let $k$ be an integer at least $2$.
If $G$ is a connected graph of order $n$ at least $k+6$
and size $m$ at most $\frac{k+3}{2}n+\frac{k-1}{2}$,
then $G$ has a minimum $k$-degenerate cut.
\end{theorem}
Similarly as in \cite{chtazh}, 
an explicit construction shows that Theorem \ref{thmmin}
is best possible.

All proofs are given in the following section.

\section{Proofs}

A {\it $k$-core} in a graph $G$ is a set $C$ of vertices of $G$ 
inducing a subgraph with no vertex of degree at most $k$, 
that is, the minimum degree $\delta(G[C])$ of the subgraph $G[C]$
of $G$ induced by $C$ is at least $k+1$.
Clearly, 
every $k$-core contains at least $k+2$ vertices, 
the only $k$-core of order exactly $k+2$ is a clique, and 
a graph $G$ is $k$-degenerate if and only if it contains no $k$-core.

We proceed to the proof of Theorem \ref{thm:size}.

\begin{proof}[Proof of Theorem \ref{thm:size}]
Let $k$, $G$, and $n$ be as in the statement.
Since $G$ has order at least $k+2$ and no $k$-degenerate cut, 
the minimum degree of $G$ is at least $k+2$.

\begin{claim}\label{claim1}
If $u$ is a vertex of $G$ of degree at most 
$k + \frac{\sqrt{k}}{5}$,
then $u$ has least 
$k-\frac{2k}{25}-\frac{2\sqrt{k}}{5}$ neighbors 
of degree at least $k+\frac{\sqrt{k}}{5}$.
\end{claim}
\begin{proof}[Proof of Claim \ref{claim1}.]
Let $u$ be a vertex of $G$ of degree $d_G(u)$ at most $k + \frac{\sqrt{k}}{5}$.
Let $R = V(G) \setminus N_G[u]$.
Let $X$ be a set of $\left\lfloor\frac{2\sqrt{k}}{5}\right\rfloor$ neighbors of $u$
minimizing the number of edges between $X$ and $R$.
Let $Y = N_G(X) \setminus N_G[u]$ and $Z = N_G(u) \setminus X$.
Since $G$ has no $k$-degenerate cut and $n>1+d_G(u)$,
there is a $k$-core $C$ in $G$ contained in $N_G(u)$.

If some vertex in $X$ has at least $\frac{\sqrt{k}}{5}$ 
neighbors in $R$,
then, by the choice of $X$, 
all vertices in $Z$ 
have at least $\frac{\sqrt{k}}{5}$ neighbors in $R$.
Now, each of the 
$$|C\cap Z|\geq |C|-|X|\geq k+2-\left\lfloor \frac{2\sqrt{k}}{5}\right\rfloor$$
vertices in $C\cap Z$ 
has degree at least $k+\frac{\sqrt{k}}{5}$,
which completes the proof of the claim.
Hence, we may assume that every vertex in $X$ has less than 
$\frac{\sqrt{k}}{5}$ neighbors in $R$, which implies
$$|Y|<\frac{\sqrt{k}}{5}\cdot |X|\leq \frac{\sqrt{k}}{5}\cdot \frac{2\sqrt{k}}{5}=\frac{2k}{25}.$$
If $Z$ contains a set $Z'$ of at least $k-\frac{2k}{25}$ vertices 
that have at least $\frac{\sqrt{k}}{5}$ neighbors in $R$,
then each of the 
\begin{eqnarray*}
|C\cap Z'|
&\geq &|C|-|N_G(u)\setminus Z'|\\
&=&|C|-|N_G(u)|+|Z'|\\
&\geq& (k+2)-\left(k+\frac{\sqrt{k}}{5}\right)+\left(k-\frac{2k}{25}\right)\\
&>& k-\frac{2k}{25}-\frac{2\sqrt{k}}{5}
\end{eqnarray*}
vertices in $C\cap Z'$ 
has degree at least $k+\frac{\sqrt{k}}{5}$,
which completes the proof of the claim.
Hence, we may assume that 
\begin{eqnarray}\label{e1}
\mbox{\it $Z$ contains less than $k-\frac{2k}{25}$ vertices 
that have at least $\frac{\sqrt{k}}{5}$ neighbors in $R$.}
\end{eqnarray}
Note that 
$$|Z|=d_G(u)-|X|
\leq k+\frac{\sqrt{k}}{5}-\left\lfloor \frac{2\sqrt{k}}{5} \right\rfloor
\leq k-\frac{\sqrt{k}}{5}+1.$$
Since $G$ has no $k$-degenerate cut
and 
$$1+|X|+|Y|+|Z|
=1+d_G(u)+|Y|
<1+k+\frac{\sqrt{k}}{5}+\frac{2k}{25}
<n,$$
the set $Y\cup Z$ is a cut, and, hence,
there is a $k$-core $C'$ in $G$ contained in $Y\cup Z$.
Since every vertex in $Z$ 
has at most $|Z|-1\leq k-\frac{\sqrt{k}}{5}$
neighbors in $Z$,
every vertex in $C'\cap Z$ 
has at least $k+1-\left(k-\frac{\sqrt{k}}{5}\right)>\frac{\sqrt{k}}{5}$ neighbors in $Y\subseteq R$.
Hence, 
the set $Z$ contains at least 
$$|C'\setminus Y|\geq |C'|-|Y|\geq k+2-\frac{2k}{25}$$
vertices that have at least 
$\frac{\sqrt{k}}{5}$ neighbors in $R$.
This contradicts \eqref{e1}, which completes the proof of the claim.
\end{proof}
Let $V_\ell$ be the set of vertices of $G$ 
of degree at least $k + \frac{\sqrt{k}}{5}$
and let  $V_s = V(G) \setminus V_\ell$.
We use a discharging argument 
to obtain a lower bound on the degree sum of $G$.
Initially, each vertex has charge equal to its degree.
Now, each vertex in $V_\ell$ sends a charge of 
$\frac{5}{38\sqrt{k}}$ to its neighbors in $V_s$.
By Claim \ref{claim1},
every vertex in $V_s$ has at least 
$k-\frac{2k}{25}-\frac{2\sqrt{k}}{5}\geq \frac{13k}{25}$
neighbors in $V_\ell$.
Therefore a vertex in $V_s$ receives from its large degree neighbors 
a total amount of charge of at least 
$\frac{5}{38\sqrt{k}}\cdot\frac{13k}{25}=\frac{13\sqrt{k}}{190}$,
meaning that it has a final charge of at least 
$k + \frac{13\sqrt{k}}{190}$.
On the other hand, a vertex in $V_\ell$ with degree $d$ 
sends a total amount of charge of at most 
$d \cdot \frac{5}{38\sqrt{k}}$, 
implying that it has a final charge of at least
$$d 
\cdot \left(1 - \frac{5}{38\sqrt{k}}\right)
\geq \left( k + \frac{\sqrt{k}}{5}\right)
\cdot \left(1 - \frac{5}{38\sqrt{k}}\right)
=k+\frac{13\sqrt{k}}{190}-\frac{1}{38}.
$$
Summing the total final charge, we obtain 
$$2m
=\sum\limits_{v \in V(G)} d_G(v) 
\geq \left(k+\frac{13\sqrt{k}}{190}-\frac{1}{38}\right)n,$$ 
which completes the proof.
\end{proof}
For a vertex $u$ in a graph $G$ and a set $U$ of vertices of $G$,
let $d_U(u)=|U\cap N_G(u)|$ denote the number of neighbors of $u$ in $U$.

\begin{proof}[Proof of Theorem \ref{thm2}]
Suppose, for a contradiction, that $G$ is a counterexample of minimum order $n$.
In particular, 
the graph $G$ has order $n\geq 5$,
no $2$-degenerate cut
but less than $\frac{27n-35}{10}$ edges.
If $n=5$, then $G$ has less than $\frac{27\cdot 5-35}{10}=10$ edges,
which implies that $G$ has two non-adjacent vertices, say $u$ and $v$.
Now, the set $V(G)\setminus \{ u,v\}$ is a $2$-degenerate cut, which is a contradiction.
Hence, we have $n\geq 6$.
If $G$ has a cut $S$ of order at most $4$, 
then it follows that $S$ has order exactly $4$ and induces a $K_4$.
Let $K$ be a component of $G-S$.
Let $G_1=G[K\cup S]$ and $G_2=G-V(K)$.
Clearly, the two graphs $G_1$ and $G_2$ have orders $n_1$ and $n_2$, 
respectively, at least $5$.
If $G_1$ or $G_2$ has a $2$-degenerate cut $S'$, then, since $S$ is complete,
the set $S'$ is also a $2$-degenerate cut in $G$, which is a contradiction.
Hence $G_1$ and $G_2$ both have no $2$-degenerate cut.
By the choice of $G$, it follows that $G_1$ and $G_2$ have at least 
$\frac{27n_1-35}{10}$ 
and 
$\frac{27n_2-35}{10}$ 
edges, respectively.
This implies that $G$ has at least 
$\frac{27n_1-35}{10}+\frac{27n_2-35}{10}-6>\frac{27n-35}{10}$
edges, which is a contradiction.
Hence, the vertex connectivity $\kappa$ of $G$ is at least $5$.
If $n\leq 17$, 
then the average degree of $G$ is less than $\frac{27\cdot 17-35}{5\cdot 17}<5$.
This implies that the minimum degree, and, hence, also the vertex connectivity of $G$
is at most $4$, which is a contradiction.
Hence, we have $n\geq 18$.
Let $u$ be a vertex of minimum degree in $G$.
Since $5\leq \kappa\leq d_G(u)<\frac{27\cdot n-35}{5\cdot n}\leq \frac{27}{5}$,
the degree of $u$ is exactly $5$.

\begin{claim}\label{claim2}
$\sum\limits_{v\in N_G(u)}d_G(v)\geq 29$.
\end{claim}
\begin{proof}[Proof of Claim \ref{claim2}.]
Let $N_G(u)=\{ v_1,v_2,v_3,v_4,v_5\}$.
Let $X$ be the set of vertices in $N_G(u)$ with at most one neighbor outside of $N_G[u]$.
Let $\ell=|X|$. 
By symmetry, we may assume that $X=\{v_1,\dots,v_\ell\}$.
If some vertex $v\in X$ has no neighbor 
outside of $N_G[u]$, 
then $N_G(u)\setminus \{ v\}$ is a cut of order $4$, 
which is a contradiction.
If two vertices $v,v'\in X$ have the same neighbor $w$
outside of $N_G[u]$, 
then $(N_G(u)\setminus \{ v,v'\})\cup\{ w\}$ is a cut of order $4$, 
which is a contradiction.
Hence, the vertices $v_1,\dots, v_\ell$ have distinct unique neighbors 
$w_1,\dots,w_\ell$ outside of $N_G[u]$, respectively.
Since $N_G(u)$ is not a $2$-degenerate cut,
the set $N_G(u)$ contains a $2$-core.
We consider two cases.

\bigskip

\noindent {\bf Case 1.} {\it $G[N_G(u)]$ has minimum degree at least $3$.}

\bigskip

\noindent This implies that at least $8$ edges have both endpoints in $N_G(u)$.
If $|X|\geq 4$, 
then $\{v_1,v_2,w_3,w_4,v_5\}$ is a $2$-degenerate cut, which is a contradiction.
If $|X|\leq 2$, then 
$$\sum\limits_{v\in N_G(u)}d_G(v)\geq 5+ 2\cdot 8+3\cdot 2+2=29.$$
Hence, we may assume that $|X|=3$.
If $v_1$ is not adjacent to $v_2$ or $v_4$, 
then $\{v_1,v_2,w_3,v_4,v_5\}$ is a $2$-degenerate cut, which is a contradiction.
By symmetry, this implies that $\{ v_1,v_2,v_3\}$ is complete
and that there are all possible $6$ edges between $\{ v_1,v_2,v_3\}$ and $\{ v_4,v_5\}$.
It follows that at least $9$ edges have both endpoints in $N_G(u)$ and, hence,
$$\sum\limits_{v\in N_G(u)}d_G(v)\geq 5+ 2\cdot 9+2\cdot 2+3=30.$$

\bigskip

\noindent {\bf Case 2.} {\it $G[N_G(u)]$ has minimum degree at most $2$.}

\bigskip

\noindent Since $N_G(u)$ contains a $2$-core,
we may assume, by symmetry, 
that $v_1$, $v_2$, $v_3$, and $v_4$ induce a $K_4$, 
and that $v_5$ has at most two neighbors in $N_G(u)$.
Note that $v_5\not\in X$.
If $d_{N_G(u)}(v_5)\geq \ell$, then 
\begin{eqnarray*}
\sum\limits_{v\in N_G(u)}d_G(v) &\geq & 
5+(12+2\cdot d_{N_G(u)}(v_5))+2(4-\ell)+\ell+(4-d_{N_G(u)}(v_5))\\
& = &  
29+d_{N_G(u)}(v_5))-\ell
\geq 29.
\end{eqnarray*}
Hence, we may assume $d_{N_G(u)}(v_5)<\ell$.
If $d_{N_G(u)}(v_5)=1$, then $\ell\geq 2$ 
and $\{w_1,v_2,v_3,v_4,v_5\}$ is a $2$-degenerate cut, which is a contradiction.
If $d_{N_G(u)}(v_5)=2$, then $\ell\geq 3$.
By symmetry, we may assume that $v_5$ and $v_1$ are adjacent. 
Again, the set $\{w_1,v_2,v_3,v_4,v_5\}$ is $2$-degenerate cut, which is a contradiction.
If $d_{N_G(u)}(v_5)=0$ and $\ell\geq 2$, then  
$\{w_1,v_2,v_3,v_4,v_5\}$ is a $2$-degenerate cut, which is a contradiction.
Hence, we may assume that $d_{N_G(u)}(v_5)=0$ and $\ell=1$.
If $v_2$ is not adjacent to $w_1$, then
$\{w_1,v_2,v_3,v_4,v_5\}$ is a $2$-degenerate cut, which is a contradiction.
Hence, by symmetry, we may assume that $w_1$ is adjacent to $v_2$, $v_3$, and $v_4$.
If $v_2$ has more than $2$ neighbors outside of $N_G[u]$, the claim follows.
Hence, by symmetry, we may assume that $v_2$, $v_3$, and $v_4$
all have exactly one neighbor $w_2$, $w_3$, and $w_4$ outside of $N_G[u]\cup \{ w_1\}$,
respectively.
If $w_2=w_3$, then $\{ w_1,w_2,v_4,v_5\}$ is a cut of order $4$,
which is a contradiction.
Hence, by symmetry, the vertices $w_2$, $w_3$, and $w_4$ are distinct.
Now, the set $\{ w_1,w_2,v_3,v_4,v_5\}$ is a $2$-degenerate cut, which is a contradiction.
This completes the proof of the claim.
\end{proof}
Let $V_i$ be the set of vertices of $G$ of degree $i$ for $i\in\{5,6,7,8\}$, and
let $V_9$ be the set of vertices of $G$ of degree at least $9$.
We use a discharging argument 
to obtain a lower bound on the degree sum of $G$.
Initially, each vertex has charge equal to its degree.
Then each vertex $u$ in $V_i$ for $i\in \{ 6,7,8,9\}$
sends a charge of $\frac{i-5}{10}$ to each neighbor.
By Claim \ref{claim2}, a vertex $u$ in $V_5$ 
receives at least $4$ times $\frac{1}{10}$ charge from its neighbors.
This implies that the final charge of $u$ is at least $\frac{27}{5}$. 
The final charge of a vertex $u$ in $V_6\cup V_7\cup V_8$ 
is at least $d_G(u)\left(1-\frac{d_G(u)-5}{10}\right)\geq \frac{27}{5}$.
The final charge of a vertex $u$ in $V_9$ is at least 
$9\left(1-\frac{2}{5}\right)=\frac{27}{5}$.
Summing the total final charge, we obtain 
$2m=\sum\limits_{v \in V(G)} d_G(v) 
\geq \frac{27n}{5},$ 
which is a contradiction and completes the proof.
\end{proof}
Our next goal is the proof of Theorem \ref{thmmin},
which requires some preparatory results.
Recall that every vertex in a minimum cut $S$ of a graph $G$ 
has a neighbor in each component of $G-S$.
For a graph $G$, let $\kappa(G)$ denote the vertex connectivity of $G$.

\begin{lemma}\label{lemma2}
Let be an integer at least $2$.
If $G$ is a connected graph of 
order $n$ at least $k+6$ vertices, 
size $m$ at most $\frac{k+3}{2}n+\frac{k-1}{2}$,
and minimum degree exactly $k+3$, 
then $G$ has a minimum $k$-degenerate cut. 
\end{lemma}
\begin{proof}
Let $G$ be as in the statement. 
If $\kappa(G)\leq k+1$, 
then every minimum cut of $G$ is $k$-degenerate. 
Hence, we may assume that $\kappa(G)\in\{k+2,k+3\}$. 
Let $X$ be the set of all vertices of minimum degree in $G$. 
We have
\begin{eqnarray*}
(k+3)n+\sum\limits_{u\in V(G)\setminus X}\Big(d_G(u)-(k+3)\Big)
& = & 2m
\leq (k+3)n+(k-1),
\end{eqnarray*}
which implies
\begin{eqnarray}\label{emin1}
n-|X|&\leq &\sum\limits_{u\in V(G)\setminus X}\Big(d_G(u)-(k+3)\Big) \leq k-1\mbox{ and, hence,}\label{emin1}\\
|X| &\geq & n-(k-1).\label{emin2}
\end{eqnarray}
Let $S$ be a minimum cut of $G$ and suppose that $S$ is not $k$-degenerate. 
Since $|S|\in\{k+2,k+3\}$,
it follows from \eqref{emin2}
that $|X\cap S|\geq |S|-(k-1)\geq 3$. 
Since $S$ contains a $k$-core, 
there is a vertex $v$ in $X\cap S$ 
such that $v$ has at least $k+1$ neighbors in $S$ 
and, hence, at most two neighbors outside of $S$. 
Since $v$ has a neighbor in each component of $G-S$,
the graph $G-S$ has exactly two components $G_1$ and $G_2$. 
We may assume that $n(G_1)\geq n(G_2)$, 
which implies $n(G_1)\geq 2$ as $n\geq k+6$.

We consider distinct cases.

\bigskip

\noindent {\bf Case 1.} {\it $\kappa(G)=k+2$.}

\bigskip

\noindent In this case, we have $G[S]=K_{k+2}$. 
Let $v\in X\cap S$. 
The vertex $v$ has exactly one neighbor $w$ in $V(G_1)$. 
If $w$ has a neighbor $v'$ in $X\cap S$ distinct from $v$,
then $w$ is the unique neighbor of $v'$ in $V(G_1)$,
which implies the contradiction that 
$(S\setminus \{ v,v'\})\cup \{ w\}$
is a cut that is smaller than $S$.
Since $X\cap S$ contains at least $3$ vertices,
the vertex $w$ has a non-neighbor in $S$.
It follows $(S\setminus \{ v\})\cup \{ w\}$ is a minimum cut that is not complete, 
and, hence, $k$-degenerate cut,
which completes the proof in this case.

\bigskip

\noindent {\bf Case 2.} {\it $\kappa(G)=k+3$ and $S$ is a $k$-core.}

\bigskip

\noindent In this case, we have $|X\cap S|\geq 4$.
Let $v_1,v_2,v_3,v_4\in X\cap S$. 
By the same argument as in Case 1, 
for each $i\in \{ 1,2,3,4\}$,
the vertex $v_i$ has exactly one neighbor $w_i$ in $V(G_1)$ and 
the vertex $w_i$ has no neighbor in $(X\cap S)\setminus \{ v_i\}$. 
In particular, the vertices $w_1,w_2,w_3,w_4$ are distinct,
which implies $n(G_1)\geq 4$. 
Counting the edges between $\{ w_1,w_2,w_3,w_4\}$ and $S$, we obtain
\begin{eqnarray*}
d_S(w_1)+d_S(w_2)+d_S(w_3)+d_S(w_4)
& \leq &
\sum\limits_{v\in S}\Big(d_G(v)-\underbrace{(k+1)}_{\leq d_S(v)}-\underbrace{1}_{\leq d_{V(G_2)}(v)}\Big)\\
&=& |S|+\sum\limits_{v\in S\setminus X}\Big(d_G(v)-(k+3)\Big)\\
&\stackrel{\eqref{emin1}}{\leq}& (k+3)+(k-1) =2k+2.
\end{eqnarray*}
By symmetry, we may assume that $d_S(w_1),d_S(w_2)\leq k$,
which implies that $w_1$ and $w_2$ both have at most $k-1$ neighbors 
in $S\setminus \{ v_1,v_2\}$.
Now, the set $(S\setminus \{ v_1,v_2\})\cup \{ w_1,w_2\}$
is a minimum cut that is $k$-degenerate,
which completes the proof in this case.

\bigskip

\noindent {\bf Case 3.} {\it $\kappa(G)=k+3$ and $S$ is not a $k$-core.}

\bigskip

\noindent Again, we have $|X\cap S|\geq 4$.
Since $S$ has order $k+3$, is not a $k$-core, but contains a $k$-core,
there is a vertex $u$ in $S$ such that $d_S(u)\leq k$
and $S\setminus \{ u\}$ is a clique.
Since each vertex in $S\setminus \{ u\}$ has 
a neighbor in $V(G_1)$,
a neighbor in $V(G_2)$, and 
$k+1$ neigbors in $S\setminus \{ u\}$,
the at least $3$ vertices in $(X\cap S)\setminus \{ u\}$ are not adjacent to $u$.
In particular, we have $d_S(u)\leq k-1$.
Let $v\in (X\cap S)\setminus \{ u\}$.
As before, it follows that the vertex $v$ has a unique neighbor $w$ in $V(G_1)$,
and $w$ has no neighbor in $(X\cap S)\setminus \{ u,v\}$.
It follows that $(S\setminus \{ v\})\cup \{ w\}$
is a minimum cut that is $k$-degenerate,
which completes the proof in this final case.
\end{proof}
We are now in a position to complete the proof of Theorem \ref{thmmin}.

\begin{proof}[Proof of Theorem \ref{thmmin}.]
Suppose, for a contradiction, that the theorem is false.
Let $G$ be a counterexample of minimum order $n$,
that is, $n\geq k+6$,
the size $m$ of $G$ is at most $\frac{k+3}{2}n+\frac{k-1}{2}$
but $G$ has no minimum $k$-degenerate cut,
which implies that $\kappa(G)\geq k+2$.
The average degree $\frac{2m}{n}$ of $G$ satisfies 
$\frac{2m}{n}\leq k+3+\frac{k-1}{n}<k+4$,
which implies $\delta(G)\leq k+3$.
Lemma \ref{lemma2} implies $\delta(G)\leq k+2$.
Since $\kappa(G)\leq \delta(G)$, we obtain $\kappa(G)=\delta(G)=k+2$.
Let $v$ be a vertex of minimum degree.
The neighborhood $S=N_G(v)$ of $v$ is a minimum cut,
which implies $G[S]=K_{k+2}$.

First, suppose that $n=k+6$. 
Let $R=\{ x,y,z\}=V(G)\setminus N_G[v]$.
The number $m'$ of edges incident with $\{ x,y,z\}$ satisfies
\begin{eqnarray}\label{eq1}
m'\leq \frac{(k+3)(k+6)}{2}+\frac{k-1}{2}-{k+3\choose 2} 
= \frac{5}{2}k+\frac{11}{2}
\stackrel{k\geq 2}{<}3k+5.
\end{eqnarray}
If $G[R]$ has at most one edge, then $m'\geq 3\delta(G)-1=3(k+2)-1$,
which contradicts \eqref{eq1}.
If $G[R]$ contains at exactly two edges, 
say $xy$ and $xz$,
then \eqref{eq1} implies that $x$ has degree $k+2$,
and $N_G(x)$ is a minimum $k$-degenerate cut,
which is a contradiction.
It follows that $R$ is complete.
By \eqref{eq1}, we may assume that 
$d_G(x)=d_G(y)=k+2$
and $d_G(z)\leq k+3$.
Since $N_G(x)$ is not a $k$-degenerate cut,
$N_G(x)$ is complete, which implies 
$N_G(x)\cap S=N_G(y)\cap S\subseteq N_G(z)\cap S$.
Now, the set $N_G(z)\cap S$ is a cut of order at most $k+1$,
which contradicts $\kappa(G)=k+2$.
Hence, we may assume that $n>k+6$.

The graph $G'=G-v$ has order $n'=n-1\geq k+6$ 
and size $m'=m-(k+2)<\frac{k+3}{2}(n-1)$.
By the choice of $G$, the graph $G'$ has a minimum $k$-degenerate cut $S'$.
Since $|S'|\leq \kappa(G')\leq \delta(G')\leq k+2=\kappa(G)$
and $S$ is complete,
the set $S'$ is also a minimum $k$-degenerate cut in $G$.
This final contradiction completes the proof.
\end{proof}
We construct graphs showing that Theorem \ref{thmmin} is best possible.
For integers $k\geq 2$ and $s\geq 3$, 
let $G$ arise from the union of $s$ disjoint $(k+2)$-cliques 
$C_0,\ldots,C_{s-1}$ arranged in a cyclic order by
\begin{itemize}
\item adding sets $M_0,\ldots,M_{s-1}$ of edges,
where $M_i$ is a matching of size $k+2$ between $C_i$ and $C_{i+1}$, and
\item adding one further vertex $u$ 
as well as all $k+2$ possible edges between $u$ and $C_0$.
\end{itemize}
See Figure \ref{fig1} for an illustration.
The only minimum cut of $G$ is $N_G(u)=C_0$, which is not $k$-degenerate.
The order of $G$ is $n=(k+2)s+1$
and the size is $m=\left(\frac{k+3}{2}n+\frac{k-1}{2}\right)+1$.

\begin{figure}[H]
    \centering
    \begin{tikzpicture}[scale=1]

    \def\l{12}    
    \def\m{6}    
    \def\r{3}    
    \def\rr{0.55} 

    \foreach \i in {1,...,\l} {
        \coordinate (C\i) at ({90 - 360/\l*(\i-1)}:\r);

        \foreach \j in {1,...,\m} {
            \coordinate (v\i-\j) at
            ($(C\i) + ({90 - 360/\m*(\j-1)}:\rr)$);
            \fill (v\i-\j) circle (0.035);
        }

        \foreach \j in {1,...,\m} {
            \foreach \k in {\j,...,\m} {
                \ifnum\j<\k
                    \draw (v\i-\j) -- (v\i-\k);
                \fi
            }
        }
    }

    \foreach \i in {1,...,\l} {
        \pgfmathtruncatemacro{\next}{mod(\i,\l)+1}
        \foreach \j in {1,...,\m} {
            \draw (v\i-\j) -- (v\next-\j);
        }
    }

    \coordinate (u) at ($(C1) + (0,1.1)$);
    \fill[red] (u) circle (0.045);
    \node[right] at (u) {$u$};

    \foreach \j in {1,...,\m} {
    \draw[thick] (u) -- (v1-\j);
    }

    \end{tikzpicture}
\caption{The constructed graph $G$ for $k=4$, $s=12$, 
and a specific choice of the matchings $M_i$.}\label{fig1}
\end{figure}
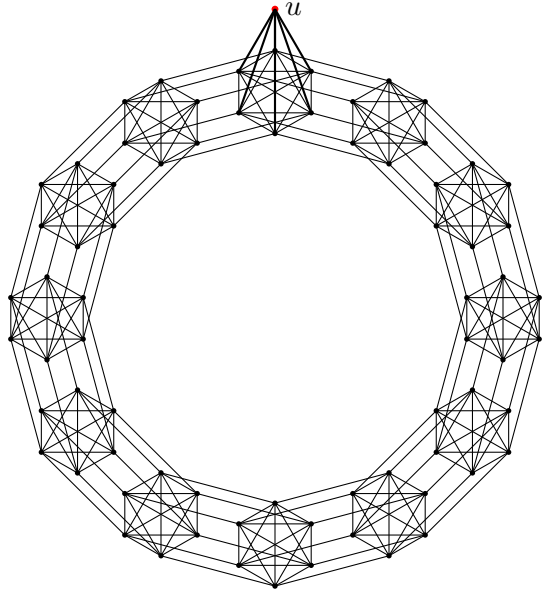

\paragraph{Acknowledgment} 
The authors were partially supported by
the Deutsche Forschungsgemeinschaft 
(DFG, German Research Foundation) -- project number 545935699.

\end{document}